\documentclass[11pt,reqno]{amsart}
\usepackage{amssymb,amsfonts,amscd,amsbsy, color, amsmath}
\usepackage[mathscr]{eucal}
\usepackage{url}
\usepackage{graphicx}
\usepackage{mathtools}
\usepackage{subcaption}
\usepackage{float}
\allowdisplaybreaks

\newtheorem{theorem}{Theorem}[section]
\newtheorem{lemma}[theorem]{Lemma}

\theoremstyle{definition}

\numberwithin{equation}{section}

\begin{document}

\title[On the 2-head of the colored Jones polynomial for pretzel knots]{On the 2-head of the colored Jones polynomial for pretzel knots}

\author{Paul Beirne}

\address{School of Mathematics and Statistics, University College Dublin, Belfield, Dublin 4, Ireland}

\email{paul.beirne@ucdconnect.ie}


\date{\today}

\begin{abstract}
   In this paper, we prove a formula for the 2-head of the colored Jones polynomial for an infinite family of pretzel knots. Following Hall, the proof utilizes skein-theoretic techniques and a careful examination of higher order stability properties for coefficients of the colored Jones polynomial.
\end{abstract}

\maketitle

\section{Introduction}

The colored Jones polynomial $J_{N,K}(q)$ of a knot $K$ is an important quantum knot invariant which, conjecturally, contains information about the geometry of $K$ \cite{Murakami}. Here, $N \in \mathbb{N}$ is the number of strands colored by the $N$-th Jones-Wenzl idempotent of the knot diagram of $K$.  Following Armond and Dasbach \cite{armond_dasbach_2011}, the tail $T_K(q)$ of the sequence $\{J_{N,K}(q)\}_{N \in \mathbb{N}}$ is a power series in $q$ such that its lowest $N$ coefficients match the lowest $N$ coefficients of $J_{N,K}(q)$ for all $N \geq 1$. Since its inception, there has been considerable interest in proving the existence (and non-existence) of the tail for various families of knots  \cite{armond_2013,armond_dasbach_2011,a-d,e-h,elhamdadi_hajij_saito_2017,gl_2016, hajij_2017,c-v,  lee_van_der_Veen_2018} and its connection to the volume conjecture \cite{dasbach_lin_2006}, quantum spin networks \cite{hajij} and Rogers-Ramanujan type identities \cite{BEIRNE2017247,keilthy_osburn_2016}.

Similarly, the head $H_K(q)$ of the sequence $\{J_{N,K}(q)\}_{N \in \mathbb{N}}$ is the power series in $q$ formed by considering the highest $N$ coefficients of $J_{N,K}(q)$ for all $N \geq 1$. Note that the colored Jones polynomial of a knot $K$ is related to the colored Jones polynomial of $-K$, the mirror of $K$, via
\begin{equation}
    J_{N,K}(q) = J_{N,-K}(q^{-1}) \nonumber
\end{equation}
\noindent and so, in particular, $H_K(q)=T_{-K}(q)$. The objective of this paper is to examine the higher order stability for the coefficients of the colored Jones polynomial for an infinite family of pretzel knots which we now describe.

 A negative twist region is a region of the knot with a positive number of negative half twists (see Figure \ref{fig:fig1}).
 
 \begin{figure}[H]
\begin{subfigure}{0.3\textwidth}
\begin{center}
\includegraphics[scale = 0.3]{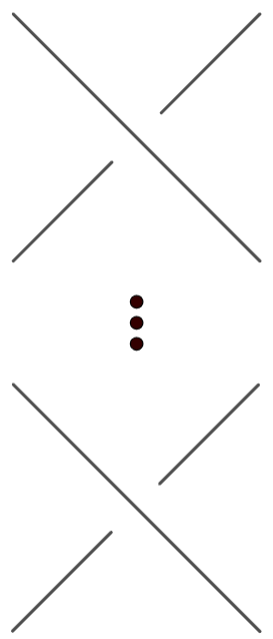}
\caption{Negative twist region}
\end{center}
\end{subfigure}
\begin{subfigure}{0.3\textwidth}
\begin{center}
\includegraphics[scale = 0.3]{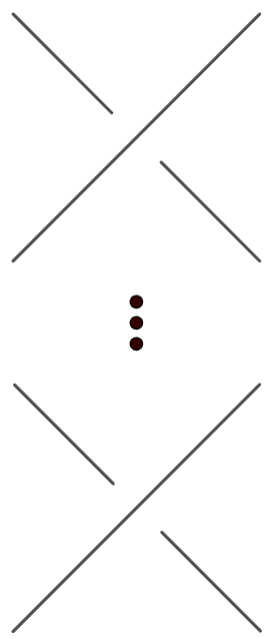}
\caption{Positive twist region}
\end{center}
\end{subfigure}
\caption{Twist regions}
\label{fig:fig1}
\end{figure}
 
 Consider the family of pretzel knots obtained by connecting three negative twist regions with strands as in Figure \ref{fig:fig2}.
 
 \begin{figure}[H]
\begin{center}
\includegraphics[scale = 0.3]{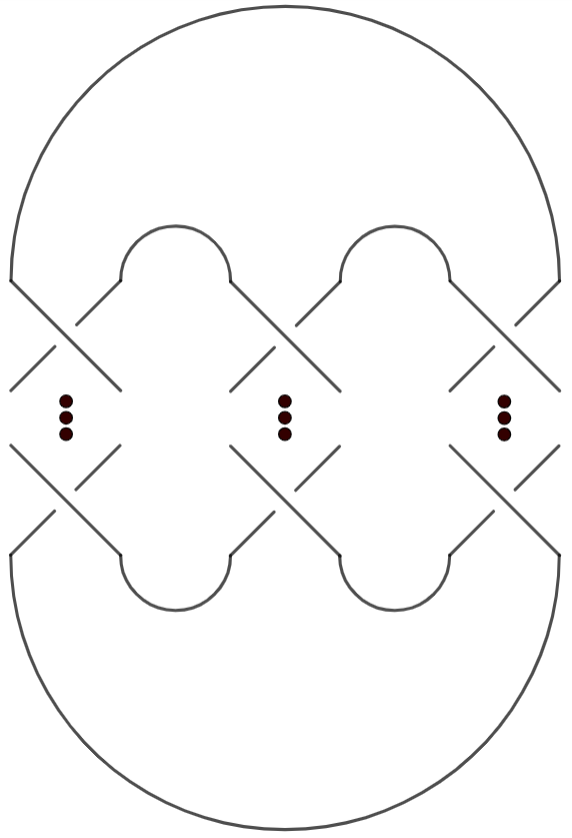} 
\end{center}
\caption{Pretzel knot with three negative twist regions}
\label{fig:fig2}
\end{figure}
 
 The Tait graph of a knot is a planar graph found by labelling each region in the knot diagram either as an $A$-region or as a $B$-region, according to the rule in Figure \ref{fig:fig3}.
 
 \begin{figure}[H]
\begin{center}
\includegraphics[scale = 0.3]{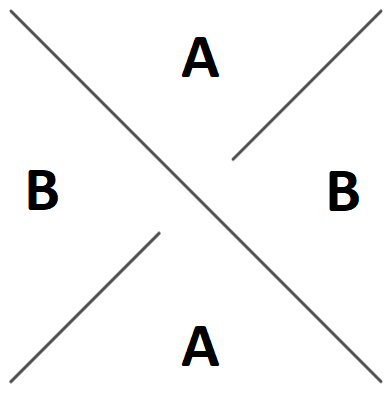} 
\end{center}
\caption{$A$ and $B$ regions}
\label{fig:fig3}
\end{figure}
 
 The knots in Figure \ref{fig:fig2} are alternating and thus such a labelling is well-defined. The $B$-checkerboard graph (or ``Tait graph") is formed by considering the $B$-regions of the knot diagram as vertices and joining two vertices by an edge for each crossing that is simultaneously adjacent to both of the corresponding regions. To construct the reduced Tait graph, replace every set of multiple edges connecting two vertices with a single edge. For example, the reduced Tait graph for the $4_1$ knot is given in Figure \ref{fig:fig4}.
 
\begin{figure}[H]
\begin{subfigure}[b]{0.3\textwidth}
\begin{center}
\includegraphics[scale = 0.4, trim = 0 2cm 0 0]{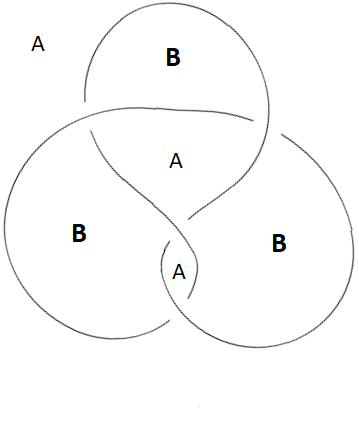}
\end{center}
\caption{Checkerboard coloring of $4_1$}
\end{subfigure}
\begin{subfigure}[b]{0.3\textwidth}
\begin{center}
\includegraphics[scale = .22]{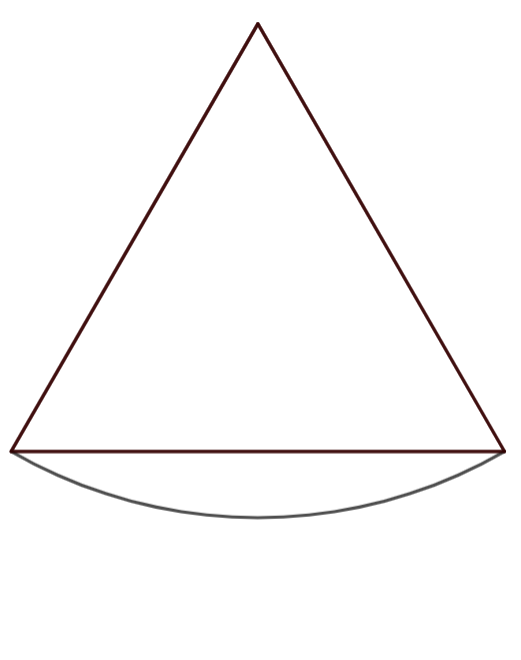}
\end{center}
\caption{Tait graph of $4_1$}
\end{subfigure}
\begin{subfigure}[b]{0.3\textwidth}
\begin{center}
    \includegraphics[scale=0.22, trim = 0 -2.5cm 0 0]{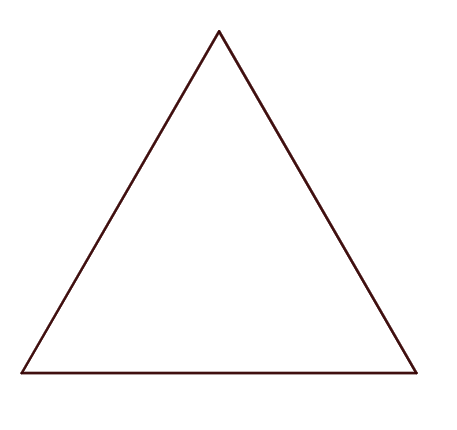}
\end{center}
\caption{Reduced Tait graph of $4_1$}
\end{subfigure}
\caption{$4_1$ knot}
\label{fig:fig4}
\end{figure}

\noindent Note that any knot in Figure 2 will have a triangle as its reduced Tait graph.

First order stability for the coefficients of the colored Jones polynomial of pretzel knots with three negative twist regions was first studied by Hall in \cite{hall_2018}. In \cite{elhamdadi_hajij_saito_2017}, Elhamdadi, Hajij and Saito proved that second order stability for this family of knots is ensured. To illustrate these stabilities, consider the highest $3N+1$ coefficients of the $N$ colored Jones polynomial for the knot $-9_{35}$.

\begin{center}
\begin{table}[h]
\centering
\begin{tabular}{|c|c c c c c c c c c c c c c c c c|}
\hline
    $H_{-9_{35}}(q)$ & $1$ & $-1$ & $-1$ & $0$ & $0$ & $1$ & $0$ & $1$ & $0$ & $0$ & $0$ & $0$ & $-1$ & $0$ & $0$ & $-1$ \\
    \hline
     $N=2$ & $1$ & $-1$ & $3$ & $-4$ & $3$ & $-5$ & $4$ & $-3$ & $2$ & $-1$ & $0$ & $0$ & $0$ & $0$ & $0$ & $0$ \\
     $N=3$ & $1$ & $-1$ & $-1$ & $4$ & $-1$ & $-6$ & $7$ & $0$ & $-11$ & $8$ & $4$ & $-13$ & $7$ & $9$ & $-13$ & $4$ \\
     $N=4$ & $1$ & $-1$ & $-1$ & $0$ & $4$ & $0$ & $-4$ & $-5$ & $7$ & $6$ & $-4$ & $-13$ & $4$ & $10$ & $3$ & $-14$ \\
     $N=5$ & $1$ & $-1$ & $-1$ & $0$ & $0$ & $5$ & $-1$ & $-3$ & $-3$ & $-6$ & $11$ & $5$ & $2$ & $-6$ & $-20$ & $8$ \\
     $N=6$ & $1$ & $-1$ & $-1$ & $0$ & $0$ & $1$ & $4$ & $0$ & $-4$ & $-3$ & $-3$ & $-2$ & $9$ & $9$ & $2$ & $-4$ \\
     $N=7$ & $1$ & $-1$ & $-1$ & $0$ & $0$ & $1$ & $0$ & $5$ & $-1$ & $-4$ & $-3$ & $-3$ & $0$ & $-3$ & $14$ & $7$ \\
     \hline
\end{tabular}
\caption{Highest coefficients of $J_{N,-9_{35}}(q)$}
\end{table}
\end{center}

In Table 1, we see that the highest $N$ coefficients of each polynomial stabilise to the highest coefficients of 
\begin{equation}
\prod^\infty_{i=1}(1-q^{-i}).
\end{equation}

Subtracting the head from each of the polynomials in Table 1 gives us the following in Table 2.
\begin{table}[h]
\begin{center}
\begin{tabular}{|c|c c c c c c c c c c c c c c c c|}
\hline
    $H_{-9_{35}}(q)$ & $1$ & $-1$ & $-1$ & $0$ & $0$ & $1$ & $0$ & $1$ & $0$ & $0$ & $0$ & $0$ & $-1$ & $0$ & $0$ & $-1$ \\
    \hline
     $N=2$ & $0$ & $0$ & $4$ & $-4$ & $3$ & $-6$ & $4$ & $-4$ & $2$ & $-1$ & $0$ & $0$ & $1$ & $0$ & $0$ & $1$ \\
     $N=3$ & $0$ & $0$ & $0$ & $4$ & $-1$ & $-7$ & $7$ & $-1$ & $-11$ & $8$ & $4$ & $-13$ & $8$ & $9$ & $-13$ & $5$ \\
     $N=4$ & $0$ & $0$ & $0$ & $0$ & $4$ & $-1$ & $-4$ & $-6$ & $7$ & $6$ & $-4$ & $-13$ & $5$ & $10$ & $3$ & $-13$ \\
     $N=5$ & $0$ & $0$ & $0$ & $0$ & $0$ & $4$ & $-1$ & $-4$ & $-3$ & $-6$ & $11$ & $5$ & $3$ & $-6$ & $-20$ & $9$ \\
     $N=6$ & $0$ & $0$ & $0$ & $0$ & $0$ & $0$ & $4$ & $-1$ & $-4$ & $-3$ & $-3$ & $-2$ & $10$ & $9$ & $2$ & $-3$ \\
     $N=7$ & $0$ & $0$ & $0$ & $0$ & $0$ & $0$ & $0$ & $4$ & $-1$ & $-4$ & $-3$ & $-3$ & $1$ & $-3$ & $14$ & $8$ \\
     \hline
\end{tabular}
\end{center}
\caption{Coefficients of $J_{N,-9_{35}}(q)$ after subtracting $H_{-9_{35}}(q)$}
\end{table}

\pagebreak

Left-justifying the coefficients in Table 2 so that each row begins with a non-zero term gives in Table 3 a stabilised sequence called the ``1-head", denoted $H_{1,-9_{35}}(q)$.

\begin{table}[h]
\begin{center}
\begin{tabular}{|c|c c c c c c c c c c c c|}
\hline
    $H_{1,-9_{35}}(q)$ & $4$ & $-1$ & $-4$ & $-3$ & $-3$ & $1$ & $0$ & $4$ & $3$ & $3$ & $3$ & $3$ \\
    \hline
     $N=2$ & $4$ & $-4$ & $3$ & $-6$ & $4$ & $-4$ & $2$ & $-1$ & $0$ & $0$ & $1$ & $0$ \\
     $N=3$ & $4$ & $-1$ & $-7$ & $7$ & $-1$ & $-11$ & $8$ & $4$ & $-13$ & $8$ & $9$ & $-13$ \\
     $N=4$ & $4$ & $-1$ & $-4$ & $-6$ & $7$ & $6$ & $-4$ & $-13$ & $5$ & $10$ & $3$ & $-13$ \\
     $N=5$ & $4$ & $-1$ & $-4$ & $-3$ & $-6$ & $11$ & $5$ & $3$ & $-6$ & $-20$ & $9$ & $6$ \\
     $N=6$ & $4$ & $-1$ & $-4$ & $-3$ & $-3$ & $-2$ & $10$ & $9$ & $2$ & $-3$ & $-12$ & $-16$ \\
     $N=7$ & $4$ & $-1$ & $-4$ & $-3$ & $-3$ & $1$ & $-3$ & $14$ & $8$ & $2$ & $-3$ & $-9$ \\
     \hline
\end{tabular}
\end{center}
\caption{Stabilising coefficients of $H_{1,-9_{35}}(q)$}
\end{table}

We observe that the highest $N-1$ coefficients of each polynomial in this new sequence stabilise to 

\begin{equation}
H_{1,-9_{35}}(q) = \prod^\infty_{i=1}(1-q^{-i})\left(1+\frac{3}{1-q^{-1}}\right),
\end{equation}

\noindent which was proven by Hall \cite{hall_2018}. Repeating this process of subtracting $H_{1,-9_{35}}(q)$ and left-justifying the coefficients, we notice that this stabilisation continues, thus forming the ``2-head" $H_{2,-9_{35}}(q)$ of $-9_{35}$ as in Table 4.

\begin{table}[h]
\begin{center}
\begin{tabular}{|c|c c c c c c|}
\hline
    $H_{2,-9_{35}}(q)$ & $-3$ & $10$ & $5$ & $-1$ & $-6$ & $-12$ \\
    \hline
     $N=2$ & $-3$ & $7$ & $-3$ & $7$ & $-5$ & $2$ \\
     $N=3$ & $-3$ & $10$ & $2$ & $-12$ & $8$ & $0$ \\
     $N=4$ & $-3$ & $10$ & $5$ & $-4$ & $-17$ & $2$ \\
     $N=5$ & $-3$ & $10$ & $5$ & $-1$ & $-9$ & $-23$ \\
     $N=6$ & $-3$ & $10$ & $5$ & $-1$ & $-6$ & $-15$ \\
     $N=7$ & $-3$ & $10$ & $5$ & $-1$ & $-6$ & $-12$ \\
     \hline
\end{tabular}
\end{center}
\caption{Stabilising coefficients of $H_{2,-9_{35}}(q)$}
\end{table}

Here, we observe that the highest $N-1$ coefficients of each polynomial in this new sequence stabilise to 

\begin{equation}\label{eqn:2head935}
H_{2,-9_{35}}(q) = \prod^\infty_{i=1}(1-q^{-i})\left(\frac{-3+10q^{-1}+5q^{-2}-4q^{-3}+q^{-4}}{(1-q^{-1})(1-q^{-2})}\right).
\end{equation}

Equation (\ref{eqn:2head935}) is one instance of our general result which we now state.

\begin{theorem}
For any pretzel knot $K$ with three negative twist regions where $n$ is the number of twist regions with exactly two half-twists and $m$ is the number of twist regions with at least three half-twists, we have

\begin{equation}
H_{2,K}(q) = \prod^\infty_{i=1}(1-q^{-i})\left(q^{-1}+n\cdot\frac{1}{(1-q^{-1})}-\frac{f_{n,m}(q)}{(1-q^{-1})(1-q^{-2})}\right)
\end{equation}

\noindent where

$$f_{n,m}(q) = \begin{cases} 
      
     0 & \text{if} \  n+m = 0, \\ 1-3q^{-1}-q^{-2}+2q^{-3} & \text{if} \  n+m=1, \\
     2-6q^{-1}-3q^{-2}+3q^{-3} & \text{if} \  n+m=2, \\
     3-9q^{-1}-6q^{-2}+3q^{-3} & \text{if} \  n+m=3. \\ 
   \end{cases}$$

\end{theorem}

This paper is organised as follows. In Section 2, we provide definitions, notation and some preliminary results necessary to prove the main result. In Section 3, we prove Theorem 1.1 by calculating an expression for the first $3N+1$ coefficients of $J_{N+1,K}(q)$ in the case where $K$ is a knot with three negative twist regions, each with at least three half-twists. We then consider cases defined by combinations of the number of crossings in each negative twist region with the previous knot as the base case.

\section{Preliminaries}

We first recall a formula for $J_{N+1,K}(q)$ as given in \cite{hall_2018}. For further details, see also \cite{masbaum_2003,masbaum_vogel_1994}. Fusion is given by

\begin{center}
\includegraphics[scale = 0.3, trim = 0 1.5cm -1cm 0]{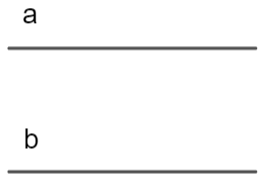} $= \sum_c\frac{\Delta_c}{\theta(a,b,c)}$\includegraphics[scale = 0.3, trim = -1cm 1.5cm 0 0]{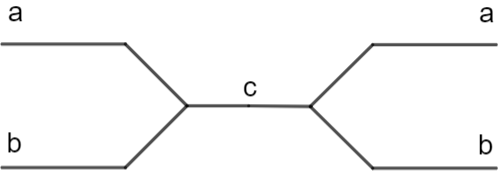}
\end{center}

\noindent where 

$$\Delta_n=\includegraphics[scale = 0.3, trim = -1cm 1.5cm -1cm 0]{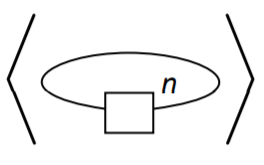} = (-1)^n[n+1],$$

\begin{equation}
[n] = \frac{\{n\}}{\{1\}}, \ \{n\}=A^{2n}-A^{-2n} \text{ and } A^{-4}=a^{-2}=q.
\label{sqbracandcurlybracdefn}
\end{equation}

\noindent Here, $[n]!$ and $\{n\}!$ are naturally defined as

\begin{equation}
[n]! = [n][n-1]\cdots[2][1], \qquad \{n\}! = \{n\}\{n-1\}\cdots\{2\}\{1\}.
\label{eq:nfactdef}
\end{equation}

\noindent Additionally, we set

\begin{equation}
    \Delta_n!=\Delta_n\Delta_{n-1}\cdots\Delta_1.
    \nonumber
\end{equation}

\noindent We also make use of the trihedron coefficient $\theta(a,b,c)$ defined by

$$\theta(a,b,c)= \includegraphics[scale = 0.3, trim = -1cm 1.5cm -1cm 0]{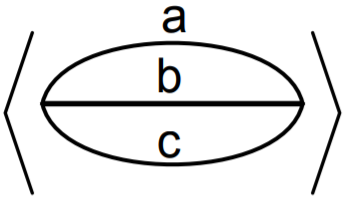} = (-1)^{i+j+k}\frac{[i+j+k+1]![i]![j]![k]!}{[i+j]![j+k]![i+k]!}$$

\noindent where

$$i=\frac{b+c-a}{2}, \qquad j=\frac{a+c-b}{2}, \qquad k=\frac{a+b-c}{2}.$$

\noindent Also, half-twists can be removed using

$$\includegraphics[scale = 0.3, trim = -1cm 1.5cm -1cm 0]{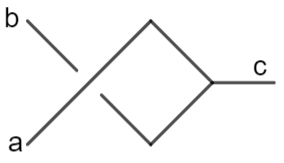} = \gamma(a,b,c)\includegraphics[scale = 0.3, trim = -1cm 1.5cm -1cm 0]{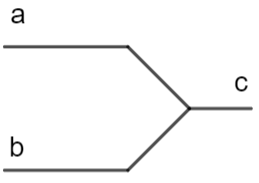}$$

\noindent where

$$\gamma(a,b,c)=(-1)^{\frac{a+b-c}{2}}A^{a+b-c+\frac{a^2+b^2-c^2}{2}}$$

\noindent is the negative half-twist coefficient.

A formula for the $N+1$ colored Jones polynomial for pretzel knots with three negative twist regions is now given by (see (4.1) in \cite{hall_2018})\footnote{We follow the convention in \cite{hall_2018} for normalisation where $J_{N,K}(q)$ is the colored Jones polynomial of $K$ with each component colored by
the $N$-dimensional irreducible representation of $\mathfrak{sl}_2$.}

\begin{equation}\label{eq:cjpfull}
J_{N+1,K}(q) = \sum_{j_1, j_2, j_3=0}^{N} S_{j_1, j_2, j_3} = \sum_{j_1, j_2, j_3=0}^{N} \prod_{i=1}^{3} \gamma(N,N,2j_i)^{m_i}\frac{\Delta_{2j_i}}{\theta(N,N,2j_i)}\Gamma_{N,(j_1,j_2,j_3)}
\end{equation}

\noindent where $m_i$ is the number of crossings in the $i$th twist region and $\Gamma_{N,(j_1,j_2,j_3)}$ is the multisum of trivalent graphs resulting from the fusion of strands and the removal of half-twists from the knot diagram (see (4.2) in \cite{hall_2018}).

Throughout this paper, we will only be concerned with a certain number of the highest coefficients of the various terms in an expression. To facilitate such calculations, we write $f(q) \overset{\cdotp n}{=} g(q)$ if the coefficients of the $n$ terms of highest degree in $f(q)$ agree with the coefficients of the $n$ terms of highest degree in $g(q)$, up to a common sign change. We write $f(q) \overset{\cdotp \infty}{=} g(q)$ if $f(q) = \pm q^s g(q)$ for $s \in \mathbb{Z}$.
Our first result is a routine extension of equation (4.15) in \cite{hall_2018}.

\begin{lemma}

For $K$ a pretzel knot with three negative twist regions, each with at least three twists, we have

\begin{equation}\label{eq:cjp}
J_{N+1,K}(q) \overset{\cdotp3N+1}{=} \frac{(-1)^N\{3N+1\}!\{N\}!^3}{\{2N\}!^3\{1\}}.
\end{equation}

\end{lemma}

We will also recall (4.22) in \cite{hall_2018}.

\begin{lemma}
We have 

\begin{equation}
    (-1)^N\{N\}! \overset{\cdotp \infty}{=}\prod^N_{i=1}(1-q^{-i}).
    \label{eqn:Nfac}
\end{equation}
\end{lemma}

In order to prove Theorem 1.1, we will also need the following lemmas which will allow us to work with the first $3N+1$ coefficients of various expressions.

\begin{lemma}

We have

\begin{equation}\label{2N} 
 \{2N\}! \overset{\cdotp3N+1}{=} (-1)^{N}\{N\}!\left(1-\frac{q^{-N-1}-q^{-2N-1}}{1-q^{-1}}+\frac{q^{-2N-2}}{2(1-q^{-1})^2}-\frac{q^{-2N-2}}{2(1-q^{-2})}\right)
\end{equation}
\noindent and
\begin{equation}\label{2Nsq}
\{2N\}!^2 \overset{\cdotp3N+1}{=} \{N\}!^2 \left(1-2\cdot\frac{q^{-N-1}-q^{-2N-1}}{1-q^{-1}}+\frac{2q^{-2N-2}}{(1-q^{-1})^2}-\frac{q^{-2N-2}}{1-q^{-2}}\right).
\end{equation}
\label{lem:2nfac}
\end{lemma}

\begin{proof}[Proof of Lemma \ref{lem:2nfac}]
By (\ref{eq:nfactdef}), dividing by an appropriate power of $q$ and keeping track of terms that do not affect the highest $3N+1$ terms (here this corresponds to powers of $q$ of degree less than $-3N$ as our maximal degree term is $1$), we obtain

\begin{align}
\{2N\}! &= \{2N\}\{2N-1\}\dots\{N+1\}\{N\}! \nonumber\\
&= (q^{\frac{-2N}{2}}-q^{\frac{2N}{2}})(q^{\frac{-(2N-1)}{2}}-q^{\frac{2N-1}{2}})\dots(q^{\frac{-(N+1)}{2}}-q^{\frac{N+1}{2}})\{N\}! \nonumber\\
&= (-1)^{N}(q^{\frac{2N}{2}}-q^{\frac{-2N}{2}})(q^{\frac{2N-1}{2}}-q^{\frac{-(2N-1)}{2}})\dots(q^{\frac{N+1}{2}}-q^{\frac{-(N+1)}{2}})\{N\}! \nonumber\\
&\overset{\cdotp3N+1}{=} (-1)^{N}\{N\}!(1-q^{-2N})(1-q^{-(2N-1)})\dots(1-q^{-N-1}) \nonumber\\
&\overset{\cdotp3N+1}{=} (-1)^{N}\{N\}!\left(1-\sum^N_{i=1}q^{-N-i}+\sum^N_{\substack{k,l=1 \\ k>l}}q^{-2N-k-l}\right) \nonumber\\
&= (-1)^{N}\{N\}!\left(1-\frac{q^{-N-1}-q^{-2N-1}}{1-q^{-1}}+\frac{\sum\limits^N_{k,l=1}q^{-2N-k-l}}{2}-\frac{\sum\limits^N_{m=1}q^{-2N-2m}}{2}\right) \nonumber\\
&= (-1)^{N}\{N\}!\Bigg[1-\frac{q^{-N-1}-q^{-2N-1}}{1-q^{-1}}+\sum^N_{l=1}\frac{q^{-2N-1-l}-q^{-3N-1-l}}{2(1-q^{-1})} \nonumber\\
& \qquad \qquad \qquad -\frac{q^{-2N-2}-q^{-3N-2}}{2(1-q^{-2})}\Bigg] \nonumber\\
&\overset{\cdotp3N+1}{=} (-1)^{N}\{N\}!\left(1-\frac{q^{-N-1}-q^{-2N-1}}{1-q^{-1}}+\sum^N_{l=1}\frac{q^{-2N-1-l}}{2(1-q^{-1})}-\frac{q^{-2N-2}}{2(1-q^{-2})}\right) \nonumber\\
&=  (-1)^{N}\{N\}!\left(1-\frac{q^{-N-1}-q^{-2N-1}}{1-q^{-1}}+\frac{q^{-2N-2}-q^{-3N-2}}{2(1-q^{-1})^2}-\frac{q^{-2N-2}}{2(1-q^{-2})}\right) \nonumber\\
&\overset{\cdotp3N+1}{=} (-1)^{N}\{N\}!\left(1-\frac{q^{-N-1}-q^{-2N-1}}{1-q^{-1}}+\frac{q^{-2N-2}}{2(1-q^{-1})^2}-\frac{q^{-2N-2}}{2(1-q^{-2})}\right).
\end{align}
Equation (\ref{2Nsq}) follows upon squaring (\ref{2N}) and simplification.
\end{proof}

Here, the fifth line of (2.7) is given by how we can take none, exactly one or exactly two of the negative powers of $q$ when multiplying the terms in the fourth line of (2.7) without affecting the highest $3N+1$ terms.
The sum involving $k$ and $l$ is realised as a sum allowing $k=l$ and $k<l$. Dividing by $2$ removes the contributions of the $k<l$ terms and halves the contributions of the $k=l$ terms. Then the $m$-sum removes the other half of the $k=l$ terms.

\begin{lemma}

We have

\begin{equation}\label{3N+1}
\{3N+1\}! \overset{\cdotp3N+1}{=} (-1)^{N+1}\{2N\}!\left(1-\frac{q^{-2N-1}}{1-q^{-1}}\right)
\end{equation}

\noindent and 

\begin{equation}\label{3N}
\{3N\}! \overset{\cdotp3N+1}{=} (-1)^{N}\{2N\}!\left(1-\frac{q^{-2N-1}}{1-q^{-1}}\right).
\end{equation}
\label{lem:3n+1fac}
\end{lemma}

\begin{proof}[Proof of Lemma \ref{lem:3n+1fac}]

We follow the proof of Lemma 2.3 to obtain

\begin{align*}
\{3N+1\}! &= \{3N+1\}\{3N\}\dots\{2N+1\}\{2N\}! \\
&= (q^{\frac{-3N-1}{2}}-q^{\frac{3N+1}{2}})(q^{\frac{-3N}{2}}-q^{\frac{3N}{2}})\dots(q^{\frac{-2N-1}{2}}-q^{\frac{2N+1}{2}})\{2N\}! \\
&= (-1)^{N+1}(q^{\frac{3N+1}{2}}-q^{\frac{-3N-1}{2}})(q^{\frac{3N}{2}}-q^{\frac{-3N}{2}})\dots(q^{\frac{2N+1}{2}}-q^{\frac{-2N-1}{2}})\{2N\}! \\
&\overset{\cdotp3N+1}{=} (-1)^{N+1}\{2N\}!(1-q^{-3N-1})(1-q^{-3N})\dots(1-q^{-2N-1}) \\
&\overset{\cdotp3N+1}{=} (-1)^{N+1}\{2N\}!\left(1-\sum^{N+1}_{i=1}q^{-2N-i}\right) \\
&= (-1)^{N+1}\{2N\}!\left(1-\frac{q^{-2N-1}-q^{-3N-2}}{1-q^{-1}}\right) \\
&\overset{\cdotp3N+1}{=} (-1)^{N+1}\{2N\}!\left(1-\frac{q^{-2N-1}}{1-q^{-1}}\right). \\
\end{align*}

A similar calculation yields (\ref{3N}).

\end{proof}

\section{Proof of Theorem 1.1}

The structure of the proof of Theorem 1.1 is as follows. We will first prove the formula for our base case $(3^+,3^+,3^+)$ in which all of the negative twist regions have at least three half-twists. In this case, after applying (\ref{eq:cjp}), (\ref{3N+1}) and (\ref{2Nsq}), we obtain the highest $3N+1$ coefficients of the normalised colored Jones polynomial

\begin{equation}
    J'_{N+1,K}(q)=J_{N+1,K}(q)\cdot\frac{(-1)^N\{1\}}{\{N+1\}}.
    \label{eqn:normcjp}
\end{equation}

We obtain the 2-head $H_{2,K}(q)$ after subtracting the highest $3N+1$ coefficients of the head $H_K(q)$ and the shifted (by $q^{-N-1}$) 1-head $H_{1,K}(q)$ from the highest $3N+1$ coefficients of $J'_{N+1,K}(q)$.

We will then split the remaining nine cases into four groups depending on the number of negative twist regions with exactly one half-twist. For simplicity, we start by choosing the case with maximal $m$ within each group. Once we have proven this case, we use the argument from Section 4.3 in \cite{hall_2018} to prove cases within the same group. Namely, each negative twist region with exactly two half-twists will contribute a summand from (\ref{eq:cjpfull}) corresponding to a $j_i$ being $N-1$ and the corresponding $m_i$ equalling two. This summand does not contribute to the highest $2N+1$ coefficients and thus contributes to $H_{2,K}(q)$ in the same way as (4.29) in \cite{hall_2018}.
We now prove Theorem 1.1 in the following five sections.

\subsection{$(3^+,3^+,3^+)$}

We proceed with the base case $(3^+,3^+,3^+)$, a pretzel knot $K$ with three negative twist regions, each with at least three half-twists.

By (2.1), Lemmas 2.1 and 2.4 and (\ref{eqn:normcjp}), we have

\begin{align}
J'_{N+1,K}(q) &\overset{\cdotp3N+1}{=} \frac{\{3N+1\}!\{N\}!^3}{\{2N\}!^3\{N+1\}} \nonumber \\
&\overset{\cdotp3N+1}{=} \frac{(-1)^{N+1}\{N\}!^3}{\{2N\}!^2(q^{\frac{-N-1}{2}}-q^{\frac{N+1}{2}})}\left(1-\frac{q^{-2N-1}}{1-q^{-1}}\right). \label{eqn:aft3n1}
\end{align}

\noindent By Lemma 4.1 in \cite{hall_2018}, Lemma 2.3 and factoring out a power of $q$, we simplify (\ref{eqn:aft3n1}) to obtain

\begin{align}
J'_{N+1,K}(q) &\overset{\cdotp3N+1}{=} \frac{(-1)^N\{N\}!}{1-q^{-N-1}}\left(1-\frac{q^{-2N-1}}{1-q^{-1}}\right)\left(\frac{1}{1-2\cdot\frac{q^{-N-1}-q^{-2N-1}}{1-q^{-1}}+\frac{2q^{-2N-2}}{(1-q^{-1})^2}-\frac{q^{-2N-2}}{1-q^{-2}}}\right) \nonumber\\
&\overset{\cdotp3N+1}{=} \frac{(-1)^N\{N\}!}{1-q^{-N-1}}\left(1-\frac{q^{-2N-1}}{1-q^{-1}}\right)\Bigg(1+\Bigg(2\cdot\frac{q^{-N-1}-q^{-2N-1}}{1-q^{-1}}-\frac{2q^{-2N-2}}{(1-q^{-1})^2}+\frac{q^{-2N-2}}{1-q^{-2}}\Bigg) \nonumber \\
& \qquad  +\left(2\cdot\frac{q^{-N-1}-q^{-2N-1}}{1-q^{-1}}-\frac{2q^{-2N-2}}{(1-q^{-1})^2}+\frac{q^{-2N-2}}{1-q^{-2}}\right)^2\Bigg) \nonumber \\
&\overset{\cdotp3N+1}{=} \frac{(-1)^N\{N\}!}{1-q^{-N-1}}\left(1-\frac{q^{-2N-1}}{1-q^{-1}}\right)\Bigg(1+2\cdot\frac{q^{-N-1}-q^{-2N-1}}{1-q^{-1}}+\frac{2q^{-2N-2}}{(1-q^{-1})^2}+\frac{q^{-2N-2}}{1-q^{-2}}\Bigg).
\end{align}

\noindent We are set to subtract the head and 1-head. As the head is $\prod^\infty_{i=1}(1-q^{-i}) \overset{\cdotp N}{=} (-1)^N\{N\}!$ and the 1-head is given as (4.30) in \cite{hall_2018}, we subtract the first $3N+1$ coefficients of these two terms (shifting the 1-head by $q^{-N-1}$) to get 

\begin{align}
H_{2,K}(q) &\overset{\cdotp3N+1}{=}  \frac{(-1)^N\{N\}!}{1-q^{-N-1}}\left(1-\frac{q^{-2N-1}}{1-q^{-1}}\right)\left(1+2\cdot\frac{q^{-N-1}-q^{-2N-1}}{1-q^{-1}}+\frac{2q^{-2N-2}}{(1-q^{-1})^2}+\frac{q^{-2N-2}}{1-q^{-2}}\right) \nonumber\\
 &  \qquad -(-1)^{N}\{3N\}! - q^{-N-1}\left((-1)^{N}\{3N\}!+\frac{3(-1)^{N}\{3N\}!}{1-q^{-1}}\right) \nonumber\\
 &\overset{\cdotp3N+1}{=}\frac{(-1)^N\{N\}!}{(1-q^{-N-1})(1-q^{-1})^3(1-q^{-2})} \bigg(1-3q^{-1}+2q^{-2}+2q^{-3}-3q^{-4}+q^{-5}+2q^{-N-1}  \nonumber\\
 & \qquad -4q^{-N-2}+4q^{-N-4}-2q^{-N-5}-3q^{-2N-1}+9q^{-2N-2}-5q^{-2N-3}-5q^{-2N-4} \nonumber\\
 &  \qquad +4q^{-2N-5}\bigg)-(-1)^{N}\{3N\}!- q^{-N-1}\left((-1)^{N}\{3N\}!+\frac{3(-1)^{N}\{3N\}!}{1-q^{-1}}\right).
\end{align}

After applying (\ref{3N}) followed by (\ref{2N}) and some routine (but tedious) calculations, we obtain

\begin{align} 
 H_{2,K}(q)& \overset{\cdotp 3N+1}{=}\frac{(-1)^N\{N\}!}{(1-q^{-N-1})(1-q^{-1})^4(1-q^{-2})}(-3q^{-2N-1}+19q^{-2N-2}-34q^{-2N-3}+14q^{-2N-4} \nonumber \\
 & \qquad +18q^{-2N-5}-20q^{-2N-6}+7q^{-2N-7}-q^{-2N-8}) \nonumber\\
 &\overset{\cdotp3N+1}{=}\frac{(-1)^N\{N\}!q^{-2N-1}}{(1-q^{-1})^4(1-q^{-2})}(-3+19q^{-1}-34q^{-2}+14q^{-3}+18q^{-4}-20q^{-5}+7q^{-6}-q^{-7}) \nonumber\\
 &\overset{\cdotp\infty}{=}\prod^\infty_{i=1}(1-q^{-i}) \left(\frac{-3+10q^{-1}+5q^{-2}-4q^{-3}+q^{-4}}{(1-q^{-1})(1-q^{-2})}\right),
 \label{eqn:-9_35}
\end{align}

\noindent as required.

Here we see that when we subtract we get $2N + 1$ copies of zero in our list of coefficients and then the first $N$ coefficients of the $2$-head are given by the first $N$ coefficients of the expression in the last line, which is independent of $N$.

We note that the proof of (\ref{eqn:-9_35}) begins with (\ref{eq:cjp}) which results from only considering the summand $S_{N,N,N}$ in (\ref{eq:cjpfull}). The remaining summands in (\ref{eq:cjpfull}) do not contribute to the first $3N+1$ coefficients as each of the $m_i$ are at least three.  For the remaining nine cases, more care is required. As lower values for one or more of the $m_i$'s occur, we consider the normalised summands 

$$\overline{S_{j_1, j_2, j_3}} = S_{j_1, j_2, j_3} \cdot \frac{(-1)^N \{ 1 \}}{\{N+1\}}$$

\noindent in the following lemmas.

\begin{lemma}
We have

$$\overline{S_{N-1,N,N}}\overset{\cdotp 2N}{=}\frac{-\prod^\infty_{i=1}(1-q^{-i})}{1-q^{-1}}\left(1+\frac{2q^{-(N+1)}}{1-q^{-1}}-\frac{q^{-(2N-1)}}{1-q^{-1}}\right).$$
\label{Sum:n-1}
\end{lemma}

\begin{proof}[Proof of Lemma \ref{Sum:n-1}]

By (\ref{eqn:Nfac}), (\ref{3N}) and Section 4.3 of \cite{hall_2018} and after simplification, we have 

\begin{align*}
    \overline{S_{N-1,N,N}} & \overset{\cdotp \infty}{=} \frac{\{3N\}!\{N\}!^3}{\{1\}\{2N-2\}!\{2N\}!^2\{2N\}} \\
    &\overset{\cdotp 2N}{=}\frac{(-1)^{N}\{N\}!^3}{\{1\}\{2N-2\}!\{2N\}!\{2N\}} \\
    &\overset{\cdotp 2N}{=}\frac{(-1)^{N}\{N\}!}{1-q^{-1}}\left(1+2q^{-(N+1)}+2q^{-(N+2)}+\dots+2q^{-(2N-2)}+q^{-(2N-1)}+2q^{-2N} \right)  \\
    &\overset{\cdotp 2N}{=}\frac{\prod^\infty_{i=1}(1-q^{-i})}{1-q^{-1}}\left(1+\frac{2q^{-(N+1)}}{1-q^{-1}}-\frac{q^{-(2N-1)}}{1-q^{-1}}\right). \\
    \end{align*}
\end{proof}

\begin{lemma}
We have 

$$\overline{S_{N-2,N,N}}\overset{\cdotp N}{=}\frac{\prod^\infty_{i=1}(1-q^{-i})}{(1-q^{-2})(1-q^{-1})}.$$
\label{sumn-2}
\end{lemma}

\begin{proof}[Proof of Lemma \ref{sumn-2}]

By (\ref{eq:cjpfull}), (\ref{eqn:Nfac}), Appendix A in \cite{hall_2018} and Lemma 14.5 in \cite{Lickorish}, we have

\begin{align*}
\overline{S_{N-2,N,N}} 
&=\frac{\{1\}(-1)^N}{\{N+1\}}\gamma(N,N,2N-4)^{m_1}\gamma(N,N,2N)^{m_2+m_3}\frac{\Delta_{2N-4}}{\theta(N,N,2N-4)}\left(\frac{\Delta_{2N}}{\theta(N,N,2N)}\right)^2 \\
& \qquad \times\left(\frac{\Delta_{3N-2}!\Delta_{N+1}!\Delta_{N-3}!^2}{\Delta_{2N-5}!\Delta_{2N-1}!^2}\right) \\
&\overset{\cdotp\infty}{=}\frac{\{1\}(-1)^N}{\{N+1\}}\frac{\Delta_{2N-4}}{\theta(N,N,2N-4)}\left(\frac{\Delta_{3N-2}!\Delta_{N+1}!\Delta_{N-3}!^2}{\Delta_{2N-5}!\Delta_{2N-1}!^2}\right) \\
&\overset{\cdotp\infty}{=}\frac{\{3N-1\}!\{2N-3\}\{N+2\}!\{N\}!^2}{\{2N\}!^2\{2N-1\}!\{2\}!\{N+1\}} \\ 
&\overset{\cdotp N}{=}\frac{(-1)^N\{N\}!}{\{2\}!} \\
&\overset{\cdotp \infty}{=}\frac{\prod^\infty_{i=1}(1-q^{-i})}{(1-q^{-2})(1-q^{-1})}. \\
\end{align*}
\end{proof}

\begin{lemma}
We have

$$\overline{S_{N-3,N,N}}\overset{\cdotp N}{=}\frac{-\prod^\infty_{i=1}(1-q^{-i})}{(1-q^{-3})(1-q^{-2})(1-q^{-1})}.$$
\label{sumn-3}
\end{lemma}

\begin{proof}[Proof of Lemma \ref{sumn-3}]

By (\ref{eq:cjpfull}), (\ref{eqn:Nfac}), (A.1) in \cite{hall_2018} and Lemma 14.5 in \cite{Lickorish}, we have

\begin{align*}
\overline{S_{N-3,N,N}}&=\frac{\{1\}(-1)^N}{\{N+1\}}\gamma(N,N,2N-6)^{m_1}\gamma(N,N,2N)^{m_2+m_3}\frac{\Delta_{2N-6}}{\theta(N,N,2N-6)} \\
& \qquad \times\left(\frac{\Delta_{2N}}{\theta(N,N,2N)}\right)^2\left(\frac{\Delta_{3N-3}!\Delta_{N+2}!\Delta_{N-4}!^2}{\Delta_{2N-7}!\Delta_{2N-1}!^2}\right) \\
&\overset{\cdotp\infty}{=}-\frac{\{3N-2\}!\{2N-5\}\{N+3\}!\{N\}!^2}{\{N+1\}\{2N\}!^2\{2N-2\}!\{3\}!} \\ 
&\overset{\cdotp N}{=}\frac{(-1)^N\{N\}!}{\{3\}!} \\
&\overset{\cdotp \infty}{=}-\frac{\prod^\infty_{i=1}(1-q^{-i})}{(1-q^{-3})(1-q^{-2})(1-q^{-1})}. \\
\end{align*}
\end{proof}

\begin{lemma}
We have 

$$\overline{S_{N,N-1,N-1}}\overset{\cdotp N}{=}\prod^\infty_{i=1}(1-q^{-i})\frac{1-q^{-(N-1)}}{(1-q^{-1})^2}.$$
\label{sumn-1n-1}
\end{lemma}

\begin{proof}[Proof of Lemma \ref{sumn-1n-1}]

Proceeding as in the proof of Lemma \ref{sumn-3}, we obtain

\begin{align*}
\overline{S_{N,N-1,N-1}} &=\frac{\{1\}(-1)^N}{\{N+1\}}\gamma(N,N,2N)^{m_1}\gamma(N,N,2N-2)^{m_2+m_3}\frac{\Delta_{2N}}{\theta(N,N,2N)} \\
& \qquad \times\left(\frac{\Delta_{2N-2}}{\theta(N,N,2N-2)}\right)^2 \left(\frac{\Delta_{N-2}}{\Delta_{N-1}}\right)^2 \left(\frac{\Delta_{3N-2}!\Delta_{N-1}!^2\Delta_{N-3}!}{\Delta_{2N-1}!\Delta_{2N-3}!^2}\right) \\
&\overset{\cdotp\infty}{=}\frac{\{1\}(-1)^N}{\{N+1\}}\left(\frac{\Delta_{2N-2}}{\theta(N,N,2N-2)}\right)^2\left(\frac{\Delta_{N-2}}{\Delta_{N-1}}\right)^2\left(\frac{\Delta_{3N-2}!\Delta_{N-1}!^2\Delta_{N-3}!}{\Delta_{2N-1}!\Delta_{2N-3}!^2}\right) \\
&=\frac{[3N-1]![N]!^3[N-1]}{[2N]^2[2N]![2N-2]!^2} \\ 
&\overset{\cdotp N}{=}\frac{(-1)^{N-1}\{N\}!\{N-1\}}{\{1\}^2} \\
&\overset{\cdotp \infty}{=}\prod^\infty_{i=1}(1-q^{-i})\frac{1-q^{-(N-1)}}{(1-q^{-1})^2}. \\
\end{align*}
\end{proof}

\subsection{$(3^+,3^+,2), (3^+,2,2), (2,2,2)$}

For any pretzel knot $K$ with three negative twist regions each with at least two half-twists, we have in (\ref{eq:cjpfull}) that $m_i$ is $2$ if the $i$th negative twist region has two half-twists and $m_i \geq 3$ otherwise. By Corollary 3.5 in \cite{hall_2018}, the maximal $q$-degree, say $a$, arises from the $j_1=j_2=j_3=N$ term in (\ref{eq:cjpfull}). It follows from Lemmas 3.2--3.4 in \cite{hall_2018} that for each $m_i$ equal to 2, the contribution from the corresponding $j_i=N-1$ term has maximal $q$-degree $a-2N-1$. No other term contributes to $H_{2,K}(q)$ as decreasing any of the $j_i$ corresponding to an $m_i \geq 3$ decreases the maximum $q$-degree by at least $3N+1$. Decreasing any $j_i$ corresponding to an $m_i=2$ from $N-1$ to $N-2$ decreases the maximum $q$-degree by either $4N$ or $4N+1$, both of which are greater than or equal to $3N+1$ for all $N \geq 1$. Any combinations of decreases in multiple $j_i$'s would further decrease the maximum $q$-degree and thus not contribute to $H_{2,K}(q)$. In total, by (\ref{eqn:-9_35})   

\begin{align*}
H_{2,K}(q)&\overset{\cdotp N}{=}H_{2,-9_{35}}(q)+nq^{2N+1}\overline{S_{N,N,N-1}} 
\end{align*}

\noindent and thus

$$H_{2,K}(q)=\prod^\infty_{i=1}(1-q^{-i}) \left(\frac{-3+10q^{-1}+5q^{-2}-4q^{-3}+q^{-4}}{(1-q^{-1})(1-q^{-2})}\right)+n\prod^\infty_{i=1}(1-q^{-i})\left(\frac{1}{1-q^{-1}}\right). $$

\subsection{$(3^+,3^+,1),(3^+,2,1),(2,2,1)$}

We first prove for $K$ a knot in the family $(3^+,3^+,1)$ that $H_{2,K}(q)$ is given by 

$$\prod^\infty_{i=1}(1-q^{-i})\left(\frac{-2+7q^{-1}+2q^{-2}-4q^{-3}+q^{-4}}{(1-q^{-1})(1-q^{-2})}\right).$$

To see this, first note that $H_{2,K}(q)$ is the series that has its first $N$ coefficients the $(2N+1)$st to $(3N+1)$st coefficients of

$$J'_{N+1,K}(q)-H_K(q)-q^{N+1}H_{1,K}.$$

Without loss of generality, take $K$ to be the $(-3,-3,-1)$ pretzel knot from the family $(3^+,3^+,1)$ and $-9_{35}$ (which is the $(-3,-3,-3)$ pretzel knot) from the family $(3^+,3^+,3^+)$. By Corollary 3.5 in \cite{hall_2018}, the only summands in (\ref{eq:cjpfull}) that contribute to $J'_{N+1,K}(q)$ yet do not contribute to the first $3N+1$ coefficients of $J'_{N+1,-9_{35}}(q)$ are those coming from decreasing the $j_i$ term, say $j_3$, corresponding to the twist region with only one half-twist. By considering the maximal $q$-degree in

$$\gamma(N,N,2j_3)\frac{\Delta_{2j_3}}{\theta(N,N,2j_3)}\Gamma_{N(N,N,j_3)},$$

\noindent where $0\leq j_3 \leq N$, we need only consider the terms corresponding to $j_3=N,N-1,N-2 \text{ and }N-3$, with respective maximal $q$-degrees $a,\ a-N-1,\  a-2N-1,\  a-3N$.

Using (\ref{eqn:-9_35}), we express $H_{2,K}(q)$ as

\begin{align*}
&H_{2,-9_{35}}(q)+q^{2N+1}H_{-9_{35}}(q)+q^{N}H_{1,-9_{35}}(q)+\overline{S_{N,N,N-1}}+\overline{S_{N,N,N-2}}+\overline{S_{N,N,N-3}}-q^{2N+1}H_K(q) \\
& \quad -q^{N}H_{1,K}(q).
\end{align*}

Applying Lemmas 3.1--3.3, the fact that $H_{-9_{35}}(q)=H_K(q)$ and (1.5) from \cite{hall_2018} yields (after simplification)

\begin{align}
    H_{2,K}(q)&\overset{\cdotp 3N+1}{=}H_{2,-9_{35}}(q)+q^{N}(H_{1,-9_{35}}(q)-H_{1,K}(q))+\overline{S_{N,N,N-1}}+\overline{S_{N,N,N-2}}+\overline{S_{N,N,N-3}} \nonumber\\
    &\overset{\cdotp N}{=}H_{2,-9_{35}}(q)+q^N\left(\frac{\prod^\infty_{i=1}(1-q^{-i})}{1-q^{-1}}-\frac{\prod^\infty_{i=1}(1-q^{-i})(q^{-(N+1)})}{(1-q^{-1})^2}\right) \nonumber \\
    & \qquad -q^N\left(\frac{\prod^\infty_{i=1}(1-q^{-i})}{1-q^{-1}}\left(1+\frac{2q^{-(N+1)}}{1-q^{-1}}-\frac{q^{-(2N-1)}}{1-q^{-1}}\right)\right) \nonumber \\
    & \qquad +\frac{\prod^\infty_{i=1}(1-q^{-i})}{(1-q^{-1})(1-q^{-2})}-\frac{q^{-(N-1)}\prod^\infty_{i=1}(1-q^{-i})}{(1-q^{-1})(1-q^{-2})(1-q^{-3})}
    \label{eq:-7_4diff_full}\\
    &\overset{\cdotp N}{=}\prod^\infty_{i=1}(1-q^{-i})\left(\frac{-2+7q^{-1}+2q^{-2}-4q^{-3}+q^{-4}}{(1-q^{-1})(1-q^{-2})}\right). \nonumber
\end{align}

Now, let $K'$ be any knot in the family $(3^+,2,1)$ and consider the link $L_5a_1$ which is $(2,2,1)$. If we do not pick up additional terms in $J'_{N+1,K'}(q)$ due to a combination of decreases in the $j_i$'s, then we can apply the same argument as in Section 3.2. A decrease from a $j_i$ corresponding to a twist region with two half-twists from $N$ to $N-1$, coupled with a decrease from a $j_i$ corresponding to a twist region with one half-twist from $N$ to $N-1$ decreases the maximal $q$-degree of the corresponding summand by at least $3N+2$. As this is the smallest possible decrease in $q$-degree due to combinations of decreases of the $j_i$'s, there is no further contribution to $H_{2,K'}(q)$ aside from those discussed in Section 3.2. Thus, we get

$$H_{2,K'}(q)=\prod^\infty_{i=1}(1-q^{-i})\left(\frac{-1+7q^{-1}+q^{-2}-4q^{-3}+q^{-4}}{(1-q^{-1})(1-q^{-2})}\right)$$

\noindent and

$$H_{2,L_5a_1}(q)=\prod^\infty_{i=1}(1-q^{-i})\left(\frac{7q^{-1}-4q^{-3}+q^{-4}}{(1-q^{-1})(1-q^{-2})}\right).$$

\subsection{$(3^+,1,1),(2,1,1)$}

Let $K$ be a knot in the family $(3^+,1,1)$. Our first step is to prove that 

$$H_{2,K}(q)=\prod^\infty_{i=1}(1-q^{-i})\left(\frac{-1+4q^{-1}-3q^{-3}+q^{-4}}{(1-q^{-1})(1-q^{-2})}\right).$$

Proceeding as above, we express $H_{2,K}(q)-H_{2,-9_{35}}(q)$ as

\begin{align}
&H_{-9_{35}}(q)-H_K(q)+q^{-(N+1)}\left(H_{1,-9_{35}}(q)-H_{1,K}(q)\right)+2q^{-(2N+1)}\Bigg[H_{2,-9_{35}}(q)-H_{2,-7_4}(q) \nonumber
\\
& \quad +q^N\left(\frac{\prod^\infty_{i=1}(1-q^{-i})}{1-q^{-1}}-\frac{q^{-(N+1)}\prod^\infty_{i=1}(1-q^{-i})}{(1-q^{-1})^2}\right)\Bigg]+\overline{S_{N,N-1,N-1}}.
\label{eq:5_2diff}
\end{align}

This is due to the fact that each twist region in $K$ that has exactly one half-twist gives us a copy of the change in $2$-head between $-7_4$ (or any knot in the family $(3^+,3^+,1)$) and $-9_{35}$, provided that we take into account the difference, shifted by an appropriate power of $q$, between $H_{1,K} - H_{1,-9_{35}}$ and $H_{1,-7_4} - H_{1,-9_{35}}$. 

We also need to consider the summand  $\overline{S_{N,N-1,N-1}}$ as a result of decreasing the $j_i$'s corresponding to the twist regions with exactly one half-twist simultaneously. Any other combinations of decreases in the $j_i$'s leads to a decrease in the maximal $q$-degree of at least $3N+2$ and thus will not contribute to $H_{2,K}(q)$.

By (\ref{eq:-7_4diff_full}), (\ref{eq:5_2diff}), Lemmas 2.3, 3.1 and 3.4, the fact that $H_{-9_{35}}(q) = H_{K}(q)$ and Theorem 1.2 in \cite{hall_2018} we have

\begin{align}
    H_{2,K}(q) &\overset{\cdotp 3N+1}{=}H_{2,-9_{35}}(q)+q^{2N+1}\Bigg[H_{-9_{35}}(q)-H_K(q)+q^{-(N+1)}\left(H_{1,-9_{35}}(q)-H_{1,K}(q)\right) \nonumber \\
    & \qquad +2q^{-(2N+1)}\Bigg[H_{2,-9_{35}}(q)-H_{2,-7_4}(q)+q^N\Bigg(\frac{\prod^\infty_{i=1}(1-q^{-i})}{1-q^{-1}} \nonumber \\
    & \qquad -\frac{q^{-(N+1)}\prod^\infty_{i=1}(1-q^{-i})}{(1-q^{-1})^2}\Bigg)\Bigg]+\overline{S_{N,N-1,N-1}}\Bigg] \\
    &\overset{\cdotp 3N+1}{=}H_{2,-9_{35}}(q)+q^{2N+1}\Bigg[2q^{-(N+1)}\left(\frac{\prod^\infty_{i=1}(1-q^{-i})}{1-q^{-1}}-\prod^\infty_{i=1}(1-q^{-i})\frac{q^{-(N+1)}}{(1-q^{-1})^2}\right) \nonumber\\
    & \qquad-2q^{-(N+1)}\left(\frac{\prod^\infty_{i=1}(1-q^{-i})}{1-q^{-1}}\left(1+\frac{2q^{-(N+1)}}{1-q^{-1}}-\frac{q^{-(2N-1)}}{1-q^{-1}}\right)\right) \nonumber\\
    &\qquad +2q^{-(2N+1)}\frac{\prod^\infty_{i=1}(1-q^{-i})}{(1-q^{-1})(1-q^{-2})} -2q^{-3N}\frac{\prod^\infty_{i=1}(1-q^{-i})}{(1-q^{-1})(1-q^{-2})(1-q^{-3})} \nonumber\\
    & \qquad +q^{-(2N+3)}\frac{\prod^\infty_{i=1}(1-q^{-i})}{(1-q^{-1})}\Bigg] \nonumber\\
    &\overset{\cdotp N}{=}\prod^\infty_{i=1}(1-q^{-i})\left(\frac{-1+4q^{-1}-3q^{-3}+q^{-4}}{(1-q^{-1})(1-q^{-2})}\right). \nonumber
\end{align}

\bigskip

Here the term with prefactor $2q^{-(2N+1)}$ in (3.8) is a result of having two twist regions with exactly one half-twist in each. This gives you a copy of the change in $2$-head between $-9_{35}$ and $-7_4$ for each region. However, this term must be compensated for by adding the difference in $1$-heads mentioned above.

As for the $4_1$ knot, which is the $(2,1,1)$ pretzel knot, we may again apply the argument for going from $-7_4$ (or any knot in the family $(3^+,3^+,1)$) to $-6_2$ (or any knot in the family $(3^+,2,1)$) in order to get

$$H_{2,4_1}(q)=\prod^\infty_{i=1}(1-q^{-i})\left(\frac{4q^{-1}-q^{-2}-3q^{-3}+q^{-4}}{(1-q^{-1})(1-q^{-2})}\right).$$

\subsection{$(1,1,1)$}

Consider $3_1$ which is the $(1,1,1)$ pretzel knot. We claim that

\begin{equation}
H_{2,3_1}=\prod^\infty_{i=1}(1-q^{-i}).
\end{equation}

The difference $H_{2,3_1}(q)-H_{2,-9_{35}}(q)$ is given by

\begin{align}
&H_{-9_{35}}(q)-H_{3_1}(q)+q^{-(N+1)}\left(H_{1,-9_{35}}(q)-H_{1,3_1}(q)\right)+3q^{-(2N+1)}\Bigg[H_{2,-9_{35}}(q)-H_{2,-7_4}(q) \nonumber
\\
& \quad +q^N\left(\frac{\prod^\infty_{i=1}(1-q^{-i})}{1-q^{-1}}-\frac{q^{-(N+1)}\prod^\infty_{i=1}(1-q^{-i})}{(1-q^{-1})^2}\right)\Bigg]+3\overline{S_{N,N-1,N-1}}. \nonumber
\end{align}

This is due to the fact that each twist region with exactly one half-twist gives us a copy of the change in $2$-head between $-7_4$ (or any knot in the family $(3^+,3^+,1)$) and $-9_{35}$, provided that we take into account the difference, shifted by an appropriate power of $q$, between $H_{1,K} - H_{1,-9_{35}}$ and $H_{1,-7_4} - H_{1,-9_{35}}$.

We also need to consider the summand $\overline{S_{N,N-1,N-1}}$ as a result of decreasing the three possible pairs of $j_i$'s. Any other combinations of decreases in the $j_i$'s leads to a decrease in maximal $q$-degree of at least $3N+2$ and thus will not contribute to $H_{2,3_1}(q)$.

Thus, along with (\ref{eq:-7_4diff_full}), (\ref{eq:5_2diff}), Lemmas 2.3, 3.1 and 3.4, the fact that $H_{-9_{35}}(q) = H_{3_1}(q)$ and Theorem 1.2 in \cite{hall_2018}, following the same argument as Section 3.4, we obtain

\begin{align*}
    H_{2,3_1}(q) &\overset{\cdotp 3N+1}{=}H_{2,-9_{35}}+q^{2N+1}\Bigg[H_{-9_{35}}(q)-H_{3_1}(q)+q^{-(N+1)}\left(H_{1,-9_{35}}(q)-H_{1,3_1}(q)\right) \\
    & \qquad +3q^{-(2N+1)}\Bigg[H_{2,-9_{35}}(q)-H_{2,-7_4}(q)+q^N\Bigg(\frac{\prod^\infty_{i=1}(1-q^{-i})}{1-q^{-1}} \\
    & \qquad -\frac{q^{-(N+1)}\prod^\infty_{i=1}(1-q^{-i})}{(1-q^{-1})^2}\Bigg)\Bigg]+3\overline{S_{N,N-1,N-1}}\Bigg] \\
    &\overset{\cdotp 3N+1}{=}H_{2,-9_{35}}+q^{2N+1}\Bigg[3q^{-(N+1)}\left(\frac{\prod^\infty_{i=1}(1-q^{-i})}{1-q^{-1}}-\prod^\infty_{i=1}(1-q^{-i})\frac{q^{-(N+1)}}{(1-q^{-1})^2}\right) \\
    & \qquad-3q^{-(N+1)}\left(\frac{\prod^\infty_{i=1}(1-q^{-i})}{1-q^{-1}}\left(1+\frac{2q^{-(N+1)}}{1-q^{-1}}-\frac{q^{-(2N-1)}}{1-q^{-1}}\right)\right) \\
    &\qquad +3q^{-(2N+1)}\frac{\prod^\infty_{i=1}(1-q^{-i})}{(1-q^{-1})(1-q^{-2})} -3q^{-3N}\frac{\prod^\infty_{i=1}(1-q^{-i})}{(1-q^{-1})(1-q^{-2})(1-q^{-3})} \\
    & \qquad +3q^{-(2N+3)}\frac{\prod^\infty_{i=1}(1-q^{-i})}{(1-q^{-1})}\Bigg] \\
    &\overset{\cdotp N}{=}\prod^\infty_{i=1}(1-q^{-i}). \\
\end{align*}

This completes the proof of Theorem 1.1.

\section*{Acknowledgements}
The author would like to thank the Irish Research Council (Grant No. GOIPG/2018/2494) for their financial support and his Ph.D. advisor Robert Osburn for helpful comments and suggestions.

\end{document}